\documentclass[twoside, 12pt]{article}
\usepackage{mathrsfs}
\usepackage{amssymb, amsmath, mathrsfs, amsthm}
\usepackage{graphicx}
\usepackage{color}
\usepackage{tikz}
\usepackage[top=2cm, bottom=2cm, left=2cm, right=2cm]{geometry}
\usepackage{float, caption, subcaption}
    \usepackage{diagbox}
\usepackage{algorithm}
\usepackage{algpseudocode}
\usepackage{amsmath}
\usepackage[colorlinks,
  linkcolor=blue,
 anchorcolor=blue,
 citecolor=blue]{hyperref} 
\input{epsf}

\newcommand{\bd}{\begin{description}}
\newcommand{\ed}{\end{description}}
\newcommand{\bi}{\begin{itemize}}
\newcommand{\ei}{\end{itemize}}
\newcommand{\be}{\begin{enumerate}}
\newcommand{\ee}{\end{enumerate}}
\newcommand{\beq}{\begin{equation}}
\newcommand{\eeq}{\end{equation}}
\newcommand{\beqs}{\begin{eqnarray*}}
\newcommand{\eeqs}{\end{eqnarray*}}

\definecolor{DarkGreen}{rgb}{0.2, 0.6, 0.3}

\usepackage[utf8]{inputenc}
\usepackage[T1]{fontenc}

\usepackage{multicol}
\usepackage{amsthm}
\usepackage{multirow}
\usepackage{makecell}

\newtheorem{theorem}{Theorem}[section]

\newtheorem{lemma}{Lemma}[section]
\newtheorem{definition}{Definition}[section]
\newtheorem{corollary}[theorem]{Corollary}
\newtheorem{case}{Case}

\newtheorem{proposition}{Proposition}[section]

\usepackage{amsmath,amssymb,amsthm}
\usepackage{cleveref} 
\theoremstyle{plain}

\theoremstyle{remark}

\begin{document}

\title{\textbf{Diagonal Ramsey numbers for wheels} \footnote{Supported by the National
Science Foundation of China (No.12061059)}
}

\author{
Maoxuan Li\footnote{School of Mathematics and Statistis, Qinghai Normal University, Xining, Qinghai 810008, China. {\tt 13956107301@163.com}} \quad
Masaki Kashima\footnote{Faculty of Science and Technology, Keio University, Kanagawa, Japan. {\tt masaki.kashima10@gmail.com}} {}\footnote{Supported by JSPS KAKENHI Grant number: JP25KJ2077} \quad
Yaping Mao \footnote{Academy of Plateau
Science and Sustainability, and School of Mathematics and Statistics, Qinghai Normal University, Xining, Qinghai 810008, China. {\tt yapingmao@outlook.com; myp@qhnu.edu.cn}}}
\date{}
\maketitle

\begin{abstract}
 The Ramsey number $\mathrm{R}(G_1,G_2)$ is the smallest integer $N$ such that any red-blue coloring of the edges of the complete graph $K_N$ contains either a red copy of $G_1$ or a blue copy of $G_2$. 
 In 2022, the third author and others gave lower and upper bounds of the Ramsey number $\mathrm{R}(W_n,W_n)$, where $W_n$ is the wheel graph with $n$ vertices.
 In this paper, we improve their bounds by showing that $3n-2\leq \mathrm{R}(W_n,W_n)\leq 6n-6$ for even $n\geq 8$ and $2n\leq \mathrm{R}(W_n,W_n)\leq \frac{9n-7}{2}$ for odd $n\geq 7$.
 Furthermore, we give recursive bounds for the $k$-colored Ramsey number for $W_n$.
\\[2mm]
\end{abstract}

\section{Introduction}

All graphs considered in this paper are finite simple graphs without loops. 
For a positive integer $n$, $K_n$ denotes the complete graph with $n$ vertices, and $C_n$ denotes the cycle of length $n$.
Other basic notation not defined in this paper is based on Bondy and Murty~\cite{bondy2008graph}.

Classical Ramsey theory was first introduced by Frank Plumpton Ramsey, a British mathematician. In his posthumous 1930 paper `On a Problem of Formal Logic', he proposed Ramsey’s Theorem as an auxiliary tool to solve a problem in formal logic. 
Though it was not originally designed for combinatorial research, this theorem later became the cornerstone of classical Ramsey theory. 
Formally, Ramsey’s Theorem states that for any given positive integer, there exists a minimum positive integer, called the Ramsey number, such that if the subsets of a set with size not less than this minimal integer are partitioned into several classes, at least one class will contain all subsets of a smaller specified substructure. 

In the setting of graph theory, the Ramsey number for graphs are defined as follows.

\begin{definition}
    For two graphs $G_1$ and $G_2$, we write $F\to (G_1,G_2)$ if for any two coloring of edges of $F$ in red and blue, either a red copy of $G_1$ or a blue copy of $G_2$ exists.
    The Ramsey number for $G_1$ and $G_2$, denoted by $\mathrm{R}(G_1,G_2)$, is defined as
    \[
    \mathrm{R}(G_1,G_2)=\min \{N\mid K_N\to (G_1,G_2)\}.
    \]
\end{definition}

It is well known that $\mathrm{R}(K_3,K_3)=6$, and the Ramsey number $\mathrm{R}(K_n,K_n)$ for general $n$ has been a main topic in graph Ramsey theory for a long time.
In this paper, we focus on the Ramsey number for a wheel graph.
For a positive integer $n\geq 4$, the \emph{wheel graph} $W_n$ is a graph obtained from the cycle $C_{n-1}$ of length $n-1$ by adding a new vertex adjacent to every vertex of $C_{n-1}$.

The study of the Ramsey numbers involving wheel graphs has a rich developmental history spanning several decades. Early contributions include Burr and Erd\H{o}s~\cite{burr1983generalizations}, who in 1983 established the exact value $\mathrm{R}(K_3, W_n) = 2n - 1$ for all $n \geq 6$ as part of their generalizations of Chv\'{a}tal's Ramsey-theoretic results. Prior to that, Harborth and Mengersen~\cite{harborth1988all} had already proven in 1988/89 that $\mathrm{R}(W_5, W_5) = 15$, a value which Hendry later listed without proof in his 1989 table of the Ramsey numbers for five-vertex graphs \cite{hendry1989ramsey}. In 1993, Faudree and McKay~\cite{faudree1993conjecture} made progress on specific wheel-wheel Ramsey numbers by characterizing all critical colorings (of types 2, 1, and 2) for $\mathrm{R}(W_n, W_6)$ when $n = 4, 5, 6$, advancing the understanding of structural properties of critical graphs in this context. Building on Burr and Erd\H{o}s' work, Radziszowski and Xia~\cite{radziszowski1994paths} extended the analysis in 1994 by determining all critical colorings for $\mathrm{R}(K_3, W_n)$ for every $n \geq 3$ in their study of paths, cycles, and wheels in graphs without antitriangles. A notable milestone in this area is the refutation of an Erd\H{o}s conjecture (labeled Erd\H{o}s in Graham, Rothschild, and Spencer's 1990 monograph \cite{graham1991ramsey}), which posited that $\mathrm{R}(G, G) \geq \mathrm{R}(K_{\chi(G)}, K_{\chi(G)})$ for any graph $G$; this conjecture was disproven by the observation that $\mathrm{R}(W_6, W_6) = 17$, while $\mathrm{R}(K_4, K_4) = 18$ and $\chi(W_6) = 4$, making $W_6$ a counterexample. 
Other results on the Ramsey number involving wheel graphs are on \cite{chen2005ramsey, salman2007ramsey, shi2010ramsey, zhang2010ramsey}.
Recently, Mao, Wang, Magnant, and Schiermeyer \cite{mao2022ramsey} provided general lower and upper bounds for diagonal wheel Ramsey numbers in 2022.

\begin{theorem}[\cite{mao2022ramsey}]\label{thm:previous_bound}
    For any integer $n \geq 7$,
    \begin{equation*}
        \begin{cases}
            3n-3 \leq \mathrm{R}(W_n,W_n)\leq 8n-10 \quad &\text{if } n \text{ is even; and}\\
            2n-2 \leq \mathrm{R}(W_n,W_n)\leq 6n-8 \quad &\text{if } n \text{ is odd.}\\
        \end{cases}
    \end{equation*}
\end{theorem}

In this paper, we improve both lower and upper bounds in Theorem~\ref{thm:previous_bound} as follows.

\begin{theorem}\label{thm:2_ramsey}
    For any integers $n\geq 7$,
    \begin{equation*}
        \begin{cases}
            3n-2 \leq \mathrm{R}(W_n,W_n)\leq 6n-6 \quad &\text{if } n \text{ is even; and}\\
            2n \leq \mathrm{R}(W_n,W_n)\leq \frac{9n-7}{2} \quad &\text{if } n \text{ is odd.}\\
        \end{cases}
    \end{equation*}
\end{theorem}

As a natural generalization of the Ramsey number for two colors, the $k$-colored Ramsey number for is defined as follows.

\begin{definition}
    For given graphs $G_1, \ldots, G_k$ and a graph $F$, we write $F \to\left(G_1, \ldots, G_k\right)$ if, no matter how one colors the edges of $F$ with $k$ colors $1, \ldots, k$, there exists a monochromatic copy of $G_i$ of color $i$ in $F$, for some $1 \leq i \leq k$. 
    The $k$-colored Ramsey number $\mathrm{R}\left(G_1, \ldots, G_k\right)$ is defined as
    \[
    \mathrm{R}\left(G_1, \ldots, G_k\right)=\min \left\{N\mid  K_N \to\left(G_1, \ldots, G_k\right)\right\}.
    \]
    In particular, when every graph in $G_1, \ldots ,G_k$ is isomorphic to a graph $G$, then we write $\mathrm{R}_k(G)$ instead of $\mathrm{R}(G_1,\ldots ,G_k)$.
\end{definition}

Note that $\mathrm{R}_1(G)=|V(G)|$ for any graph $G$.
When $G_1=\cdots =G_k=P_3$, where $P_3$ is a path with three vertices, $F \to\left(G_1, \ldots, G_k\right)$ if and only if $F$ has a proper edge coloring with $k$ colors.
Thus, we have the exact value $\mathrm{R}_k(P_3)=k$ in this case.
On the other hand, when it comes to other graph $G$, it is difficult to determine the value of $\mathrm{R}_k(G)$.
Indeed, even for the case of $G=K_3$, the asymptotic behavior of $\mathrm{R}_k(K_3)$ is unknown, while it was shown by Greenwood and Gleason~\cite{greenwood1955combinatorial} that $\mathrm{R}_3(K_3)=17$ and $42\leq \mathrm{R}_4(K_3)\leq 66$.

In this paper, we show recursive relations on the $k$-colored Ramsey number $\mathrm{R}_k(W_n)$ for wheels that give lower and upper bounds of $\mathrm{R}_k(W_n)$.
Let $K_4^-$ denote the graph obtained from $K_4$ by deleting one edge.

\begin{theorem}\label{thm:k_ramsey_lower}
    Let $k\geq 2$ and $n\geq 2$ be integers.
    Then
    \[
    \mathrm{R}_k(W_n)-1\geq \max\{(\mathrm{R}_{k-\ell}(K_4^-)-1)\cdot (\mathrm{R}_{\ell}(W_n)-1)\mid 1\leq \ell\leq k-1\}
    \]
    for even $n$, and 
    \[
    \mathrm{R}_k(W_n)-1\geq \max\{(\mathrm{R}_{k-\ell}(K_3)-1)\cdot (\mathrm{R}_{\ell}(W_n)-1)\mid 1\leq \ell\leq k-1\}
    \]
    for odd $n$.
\end{theorem}

By Theorems~\ref{thm:2_ramsey} and \ref{thm:k_ramsey_lower}, we have the following lower bound of $\mathrm{R}_k(W_n)$.

\begin{corollary}\label{cor:k_ramsey_lower}
    For positive integers $k\geq 2$ and $n\geq 4$,
    \[
    \begin{cases}
        \mathrm{R}_k(W_n)\geq 3^{k-1}(n-1)+1\quad &\text{if } n \text{ is even; and}\\
        \mathrm{R}_k(W_n)\geq 2^{k-2}(2n-1)+1\quad &\text{if } n \text{ is odd.}
    \end{cases}
    \]
\end{corollary}

\begin{proof}
    The proof goes by induction on $k$.
    When $k=2$, we have $\mathrm{R}_2(W_n)\geq 3n-2$ for even $n$ and $\mathrm{R}_2(W_n)\geq 2n$ for odd $n$ by Theorem~\ref{thm:2_ramsey}.
    Suppose that the statement holds for $k-1$.
    Since $\mathrm{R}_1(K_4^-)=4$ and $\mathrm{R}_1(K_3)=3$, Theorem~\ref{thm:k_ramsey_lower} implies that
    \[
    \mathrm{R}_k(W_n)\geq (\mathrm{R}_1(K_4^-)-1)\cdot (\mathrm{R}_{k-1}(W_n)-1)+1\geq 3\cdot 3^{k-2}(n-1)+1=3^{k-1}(n-1)+1
    \]
    for even $n$, and that
    \[
    \mathrm{R}_k(W_n)\geq (\mathrm{R}_1(K_3)-1)\cdot (\mathrm{R}_{k-1}(W_n)-1)+1\geq 2\cdot 2^{k-3}(2n-1)+1=2^{k-2}(2n-1)+1
    \]
    for odd $n$, as desired.
\end{proof}

The following recursive formula gives an upper bound of $\mathrm{R}_k(W_n)$.

\begin{theorem}\label{thm:k_ramsey_upper}
    For positive integers $k\geq 2$ and $n\geq 4$, 
    it follows that
    \[
    \mathrm{R}_k(W_n)\leq k\bigl(\mathrm{R}(C_{n-1}, \overbrace{W_n, \ldots , W_n}^{k-1})-1\bigr)+2 \leq k\bigl(\mathrm{R}(C_{n-1},K_{\mathrm{R}_{k-1}(W_n)})-1\bigr)+2.
    \]
    Furthermore, 
    \[
    \mathrm{R}_k(W_n)\leq k\left(\frac{n-3}{2}\mathrm{R}_{k-1}(W_n)\cdot (\mathrm{R}_{k-1}(W_n)-1)+\mathrm{R}_{k-1}(W_n)-1\right)+2.
    \]
\end{theorem}

The following classic result by Erd\H{o}s et al.~\cite{erdos1978cycle} gives an asymptotic bound of the Ramsey number $\mathrm{R}(C_{\ell}, K_s)$.

\begin{theorem}[\cite{erdos1978cycle}]\label{thm:cycle_clique_ramsey_asympt}
    For every fixed integer $\ell$, there exists a constant $a_{\ell}>0$ only depends on $\ell$ such that
    \[
    \mathrm{R}(C_{\ell},K_s)\leq a_{\ell} s^{1+\frac{1}{m}},
    \]
    where $m=\left\lceil\frac{\ell}{2}\right\rceil-1$.
\end{theorem}

By Theorems~\ref{thm:k_ramsey_upper} and \ref{thm:cycle_clique_ramsey_asympt}, we obtain the following asymptotic recursive upper bound of $\mathrm{R}_k(W_n)$.

\begin{corollary}\label{cor:k_ramsey_asympt}
    For integers $k\geq 2$ and $n\geq 4$, there exist positive constants $c_n$ depending only on $n$ such that
    \[
    \mathrm{R}_k(W_n)\leq k\left(c_n\mathrm{R}_{k-1}(W_n)^{1+\varepsilon_n}-1\right)+2,
    \]
    where $\varepsilon_n=\frac{2}{n-2}$ for even $n$ and $\varepsilon_n=\frac{2}{n-3}$ for odd $n$.
\end{corollary}

\begin{proof}
    We set $c_n=a_{n-1}$, where $a_{n-1}$ is the constant in Theorem~\ref{thm:cycle_clique_ramsey_asympt}.
    Since $\varepsilon_n=\frac{2}{n-2}=1/\left(\left\lceil\frac{n-1}{2}\right\rceil-1\right)$ for even $n$ and $\varepsilon_n=\frac{2}{n-3}=1/\left(\left\lceil\frac{n-1}{2}\right\rceil-1\right)$ for odd $n$, Theorems~\ref{thm:k_ramsey_upper} and \ref{thm:cycle_clique_ramsey_asympt} imply that
    \[
    \mathrm{R}_k(W_n)\leq k\bigl(\mathrm{R}(C_{n-1},K_{\mathrm{R}_{k-1}(W_n)})-1\bigr)+2\leq k\bigl(c_n\mathrm{R}_{k-1}(W_n)^{1+\varepsilon_n}-1\bigr)+2.
    \]
\end{proof}

Since $1+\varepsilon_n<2$ for $n\geq 6$, the bound in Corollary~\ref{cor:k_ramsey_asympt} is better than that in Theorem~\ref{thm:k_ramsey_upper} for sufficiently large $n$.

In the rest of the paper, we will give a proof of  Theorem~\ref{thm:2_ramsey} in Section~\ref{section:2_colors}, and proofs of Theorems~\ref{thm:k_ramsey_lower} and \ref{thm:k_ramsey_upper} in Section~\ref{section:many_colors}.

\section{Two colors}\label{section:2_colors}

In this section, we give lower and upper bounds of the Ramsey number $\mathrm{R}(W_n,W_n)$.
Let $G$ be a graph with an edge coloring with two colors red and blue.
For a set of vertices $S\subseteq V(G)$, let $G_r[S]$ (resp. $G_b[S]$) denote the graph whose vertex set is $S$ and edge set is all the red (resp. blue) edges joining two vertices of $S$.
In particular, we use $G_r$ and $G_b$ to denote $G_r[V(G)]$ and $G_b[V(G)]$, respectively.
Note that for a vertex $v$, $N_{G_r}(v)$ ($N_{G_b}(v)$) is the set of vertices that are adjacent to $v$ in red (resp. blue) edges in $G$.

We will show three propositions which obviously yield Theorem~\ref{thm:2_ramsey}.

\subsection{Lower bound}

We will show a lower bound of $\mathrm{R}(W_n,W_n)$ in Theorem~\ref{thm:2_ramsey}.
The constructions of edge colorings to show the bound are completely different according to the parity of $n$, so we give separated proofs.

\begin{proposition}\label{prop:even_wheel_lower_bound}
    For every even integer $n\geq 4$,
    \[
    \mathrm{R}(W_n,W_n)\geq 3n-2.
    \]
\end{proposition}

\begin{proof}
    Let $G$ be the complete graph with $3n-3$ vertices, and let $A\cup B\cup C$ be a vertex partition such that $|A|=|B|=|C|=n-1$.
    We color the edges of $G$ so that all edges in $G[A]$, $G[B]$, and $G[C]$ are blue and all other edges are red.
    
    Then, since $G_b$ is isomorphic to a disjoint union of three complete graphs of order $n-1$, $G$ has no blue $W_n$.
    On the other hand, since $G_r$ is isomorphic to the complete 3-partite graph $K_{n-1,n-1,n-1}$, for every vertex $v$ of $G$, $N_{G_r}(v)$ induces a complete bipartite graph $K_{n-1,n-1}$ in $G_r$.
    As $n-1$ is an odd integer, $G_r$ does not contain a copy of $W_n$. This proves Proposition~\ref{prop:even_wheel_lower_bound}.
\end{proof}

\begin{proposition}\label{prop:lower_bound odd}
    For all odd integers $n \geq 5$,
    \[
        \mathrm{R}(W_n, W_n) \geq 2n.
    \]
\end{proposition}

\begin{proof}
    Let $G$ be the complete graph with $2n-1$ vertices, and let $A \cup  B \cup C \cup D \cup \{v_0\}$ be a vertex partition such that $|A| = |B| = |C| = |D| = \frac{n-1}{2}$.
    We color the edges of $G$ so that all edges in $G[A], G[C]$ and all edges in $E_G(A,B)\cup E_G(B,D)\cup E_G(D,C)\cup E_G(v_0, A\cup C)$ are blue and all other edges are red.
    We shall show that neither a red $W_n$ nor a blue $W_n$ exists in this colored graph $G$. 
    As the roles of the colors red and blue are symmetric, it suffices to show that $G$ has no red $W_n$.

    Suppose that $\{v\} \cup W$ forms a red $W_n$ with the center $v$ in $G$. 
    Note that $G_r[W]$ contains a cycle of length $n-1$ in $G$. 
    Without loss of generality, we may assume that $v \in \{v_0\}\cup A\cup B$.

    If $v=v_0$, then we have that $W \subseteq N_{G_r}(v_0) = B\cup D$.
    However, since $G_r[B\cup D]$ is the disjoint union of complete graphs $G_r[B]$ and $G_r[D]$ of size $\frac{n-1}{2}$, it does not contain a cycle of length $n-1$, a contradiction.

    If $v \in  A$, then we have that $W \subseteq N_{G_r}(v) = C\cup D$. However, since $G_r[C\cup D]$ is the disjoint union of the independent vertices $G_r[C]$ of size $\frac{n-1}{2}$ and the complete graph $G_r[D]$ of size $n-1$, it does not contain a cycle of length $n-1$, a contradiction.

    If $v \in B$, then we have that $W \subseteq N_{G_r}(v) = (B\setminus \{v\}) \cup \{v_0\} \cup C$. 
    Note that $|(B\setminus \{v\})\cup \{v_0\}\cup C| = n-1$. 
    Since $G_r[(B\setminus \{v\}) \cup \{v_0\} \cup C]$ is the join of the complete graph $G_r[B\setminus \{v\}]$ of size $\frac{n-1}{2}-1$ and the independent vertices $G_r[\{v_0\}\cup C]$ of size $\frac{n-1}{2}+1$, it is not $1$-tough. Thus, no Hamilton cycle (equivalently, a cycle of length $n-1$) exists in $G_r[(B\setminus \{v\})\cup \{v_0\}\cup C]$, a contradiction.

    Combining these three cases, we conclude that $G$ does not contain a red $W_n$. This proves Proposition~\ref{prop:lower_bound odd}.
\end{proof}

\subsection{Upper bound}

In order to improve the upper bound of $\mathrm{R}(W_n,W_n)$, we use following results.

\begin{theorem}[Diagonal Cycle Ramsey Number \cite{faudree1974all, karolyi2001generalized, rosta1973ramsey}]\label{thm:cycle_ramsey}
    For integers $k\geq 3$:
    \[
        \mathrm{R}(C_k, C_k) = 
        \begin{cases}
            2k - 1 & \text{if } k \text{ is odd}; \\
            6 & \text{if } k = 4 \quad (\text{exceptional case}); \\
            \frac{3k}{2} - 1 & \text{if } k \text{ is even and } k \geq 6.
        \end{cases}
    \]
\end{theorem}

Our key idea is the use of a result on pancyclicity of graphs by Brandt et al.~\cite{brandt1998weakly}.
A graph is said to be \emph{weakly pancyclic} if it contains a cycle of every length between its girth and circumference.

\begin{theorem}[\cite{brandt1998weakly}]\label{thm:weakly_pancyclic}
    Every non-bipartite graph $G$ of order $n$ with
    $\delta(G)\geq \frac{n+2}{3}$ is weakly pancyclic and has girth 3 or 4, where $\delta(G)$ is the minimum degree of $G$.
\end{theorem}

By using a well known Dirac's theorem on Hamiltonicity of a graph~\cite{dirac1952some} and its variation, we can show the following.

\begin{lemma}\label{lem:circumference}
    Let $G$ be a graph of order $n$ with $\delta(G)\geq \frac{n+2}{3}$.
    Then $G$ has circumference at least $\frac{n-1}{2}$.
\end{lemma}

\begin{proof}
    When $G$ is 2-connected, it is well known that $G$ contains a cycle of length at least $2\delta(G)\geq \frac{2(n+2)}{3}>\frac{n-1}{2}$, as desired.
    Thus, we assume that $G$ has a vertex $v$ such that $G-v$ is disconnected. 
    Note that it is possible that $G$ itself is disconnected.
    Let $G_1$ be a component of $G-v$ and let $G_2=G-\{v\}\cup V(G_1)$.
    Without loss of generality, we may assume that $|V(G_1)|\leq |V(G_2)|$.
    Considering a vertex $u\in V(G_1)$, we have $\frac{n+2}{3}\leq d_G(u)\leq |\{v\}\cup (V(G_1)\setminus \{u\})|=|V(G_1)|$, and hence $|V(G_2)|=n-1-|V(G_1)|\leq n-1-\frac{n+2}{3}=\frac{2n-5}{3}$.
    Since $\delta(G_2)\geq \delta(G)-1\geq \frac{n-1}{3}>\frac{1}{2}|V(G_2)|$, by Dirac's theorem, $G_2$ is Hamiltonian.
    This implies that $G$ contains a cycle of length at least $|V(G_2)|\geq \frac{n-1}{2}$.
\end{proof}

Now we prove the upper bound of $\mathrm{R}(W_n,W_n)$, as follows.

\begin{proposition}\label{prop:two_colors_upper_bound}
    For all integers $n\geq 7$,
    \[
    \mathrm{R}(W_n, W_n) \leq
    \begin{cases}
        6n-6 & \text{if } n \text{ is even}, \\
        \frac{9n-7}{2} & \text{if } n \text{ is odd}.
    \end{cases}
    \]
\end{proposition}

\begin{proof}
    Let $G$ be a complete graph of order $N$ with any red-blue coloring of the edges.
    We fix a vertex $v_0$ of $G$.
    Since $|N_{G_r}(v_0)|+|N_{G_b}(v_0)|=N-1$, without loss of generality, we may assume that $|N_{G_r}(v_0)|\geq \left\lceil\frac{N-1}{2}\right\rceil$.
    We set $S:=N_{G_r}(v_0)$ and $s:=|S|$.
    If there is a vertex $v$ of $S$ such that $|N_{G_b}(v)\cap S|\geq \mathrm{R}(C_{n-1},C_{n-1})$, then $G[N_{G_b}(v)\cap S]$ contains either a red $C_{n-1}$ or a blue $C_{n-1}$, which together with $v_0$ or $v$ forms a red $W_n$ or a blue $W_n$, respectively.
    Thus, we may assume that $|N_{G_b}(v)\cap S|\leq \mathrm{R}(C_{n-1},C_{n-1})-1$, and hence $d_{G_r}(v)\geq s-1-|N_{G_b}(v)\cap S|\geq s-\mathrm{R}(C_{n-1},C_{n-1})$.
    As the choice of $v\in S$ is arbitrary, we conclude that $\delta(G_r[S])\geq s-\mathrm{R}(C_{n-1},C_{n-1})$.

    We divide the argument according to the parity of $n$, and show that $G_r[S]$ contains a cycle of length $n-1$.

    \setcounter{case}{0}
    \begin{case}
        $n$ is even.
    \end{case}

    Let $N=6n-6$ and we shall show that $G$ contains either a red $W_n$ or a blue $W_n$.
    We have $s\geq \left\lceil\frac{N-1}{2}\right\rceil\geq \left\lceil\frac{6n-7}{2}\right\rceil=3n-3$.
    By Theorem~\ref{thm:cycle_ramsey}, we have $\mathrm{R}(C_{n-1},C_{n-1})=2(n-1)-1=2n-3$, and hence $\delta(G_r[S])\geq s-2n+3$.
    Since $s-2n+3\geq \frac{s+2}{3}$ holds for $s\geq 3n-\frac{7}{2}$, we have
    \[\delta(G_r[S])\geq s-2n+3\geq \frac{s+2}{3}=\frac{|S|+2}{3}.\]
    
    Suppose that $G_r[S]$ is a bipartite graph with parts $S_1$ and $S_2$.
    Then one of them, say $S_1$, has order at least $\frac{s}{2}\geq \frac{3n-3}{2}$.
    Since $G_b[S_1]$ is a complete graph of order at least $\frac{3n-3}{2}\geq n$ for $n\geq 3$, $G$ contains a blue $W_n$.

    Thus, we may assume that $G$ is non-bipartite.
    By Theorem~\ref{thm:weakly_pancyclic} and Lemma~\ref{lem:circumference}, $G_r[S]$ is weakly pancyclic and has girth at most $4\leq n-1$ and circumference at least $\frac{s-1}{2}\geq \frac{3n-4}{2}\geq n-1$.
    Thus, $G_r[S]$ contains $C_{n-1}$, which together with $v_0$ yields a red $W_n$ in $G$.

    \begin{case}
        $n$ is odd.
    \end{case}

    Let $N=\frac{9n-7}{2}$ and we shall show that $G$ contains either a red $W_n$ or a blue $W_n$.
    Using a similar argument as in Case 1, by Theorem~\ref{thm:cycle_ramsey}, we have $s\geq \left\lceil\frac{N-1}{2}\right\rceil\geq \left\lceil\frac{9n-9}{4}\right\rceil\geq \frac{9n-9}{4}$ and 
    $\delta(G_r[S])\geq s-\mathrm{R}(C_{n-1},C_{n-1})= s-\frac{3n-5}{2}$.
    Since $s-\frac{3n-5}{2}\geq \frac{s+2}{3}$ holds for $s\geq \frac{9n-11}{4}$, we have
    \[\delta(G_r[S])\geq s-\frac{3n-5}{2}\geq \frac{s+2}{3}=\frac{|S|+2}{3}.\]
    
    If $G_r[S]$ is bipartite, then as in Case 1, we conclude that $G_b[S]$ contains a complete graph of order at least $\frac{s}{2}\geq \frac{9n-7}{8}\geq n$ for $n\geq 7$, and hence $G$ contains a blue $W_n$.
    Otherwise, by Theorem~\ref{thm:weakly_pancyclic} and Lemma~\ref{lem:circumference}, $G_r[S]$ is weakly pancyclic and has girth at most $4\leq n-1$ and circumference at least $\frac{s-1}{2}\geq \frac{9n-13}{8}\geq n-1$.
    Thus, $G_r[S]$ contains $C_{n-1}$, which together with $v_0$ forms a red $W_n$ in $G$.

    Combining Cases 1 and 2, we conclude that $\mathrm{R}(W_n,W_n)\leq 6n-6$ for even $n$ and $\mathrm{R}(W_n,W_n)\leq \frac{9n-7}{2}$ for odd $n$.
\end{proof}

\section{More colors}\label{section:many_colors}

In this section, we give proofs of Theorems~\ref{thm:k_ramsey_lower} and \ref{thm:k_ramsey_upper}.

\subsection{Lower bound}

\begin{proof}[Proof of Theorem~\ref{thm:k_ramsey_lower}]
    Let $k$ and $\ell$ be positive integers with $k\geq 2$ and $1\leq \ell\leq k-1$.
    The argument is almost same despite the parity of $n$, so we only show the case $n$ is even. (In fact, the case $n$ is odd is simpler.)

    Suppose that $n$ is even.
    Set $s:=\mathrm{R}_{\ell}(K_4^-)-1$, $t:=\mathrm{R}_{k-\ell}(W_n)-1$, and $N:=st$.
    Let $G$ be a graph isomorphic to $K_N$, and we shall construct an edge coloring of $G$ with $k$ colors without monochromatic $W_n$.
    By the definition of $s$, the complete graph $K_s$ with vertices $\{v_1,v_2,\ldots ,v_s\}$ has an edge coloring $\chi'$ with colors $\{1,\ldots ,\ell\}$ without monochromatic $K_4^-$. 
    Let $V_1\cup V_2 \cup \cdots \cup V_s$ be a partition of the vertices of $K_N$ such that $|V_i| = t$ for any $i \in  \{1,\ldots ,s\}$. 
    We define an edge coloring $\chi$ of $G$ as follows:
    \begin{itemize}
        \item For any $i \in \{1, \ldots, s\}$, as $|V_i| < \mathrm{R}_{k-\ell}(W_n)$, we color the edges of $G[V_i]$ in colors $\{\ell + 1, \ldots, k\}$ so that no monochromatic $W_n$ exists.
        \item For any distinct $i, j \in \{1, \ldots, s\}$ and any $x\in V_i$, $y\in V_j$, we let $\chi(xy) = \chi'(v_iv_j)$.
    \end{itemize}
    By the definition, it is easy to see that $G$ has no monochromatic $W_n$ of a color in $\{1, \ldots , s\}$.
    Suppose that $\{v\}\cup W$ forms a monochromatic $W_n$ with the center $v$. 
    Let $c$ be the color of the monochromatic $W_n$, and let $G_c$ be the graph whose vertex set is $V(G)$ and edge set is all the edges of color $c$.
    As $n$ is even, $G_c[W]$ contains an odd cycle.
    On the other hand, since $\chi'$ is an edge coloring without monochromatic $K_4^-$, $G_c[N_{G_c}(v)]$ is a bipartite graph containing $W$, a contradiction.
    Thus, no monochromatic $W_n$ exists with respect to $\chi$, and hence $\mathrm{R}_k(W_n)\geq N+1=st+1=(\mathrm{R}_{\ell}(K_4^-)-1)\cdot (\mathrm{R}_{k-\ell}(W_n)-1)$, as desired.
    
    We remark that when $n$ is odd, we set $\chi'$ as an edge coloring of $K_{\mathrm{R}_{\ell}(K_3)-1}$ without monochromatic triangle, and then $\chi$ has no monochromatic triangle of a color in $\{1,\dots ,\ell\}$ neither.
\end{proof}

\subsection{Upper bound}

The following result on the Ramsey number for a tree and a complete graph by Chv\'{a}tal~\cite{chvatal1977tree} is well-known.

\begin{lemma}[\cite{chvatal1977tree}]\label{lem:tree_complete_ramsey}
    For any tree $T_s$ with $s$ vertices, we have 
    \[\mathrm{R}(T_s,K_t)=(s-1)(t-1)+1.\]
\end{lemma}

Using Lemma~\ref{lem:tree_complete_ramsey}, we show the following bound of $\mathrm{R}(C_{n-1},K_m)$, which will be used in our proof of Theorem~\ref{thm:k_ramsey_upper}.

\begin{lemma}\label{lem:cycle_complete_ramsey}
    For integers $n\geq 4$ and $m\geq 2$,
    \[
    \mathrm{R}(C_{n-1}, K_m) \leq m + \frac{n - 3}{2} m(m - 1).
    \]
\end{lemma}

\begin{proof}
    The proof goes by induction on $m$.
    For $m=2$, we have $\mathrm{R}(C_{n-1},K_2)=n-1=2+\frac{n-3}{2}\cdot 2$, hence the statement holds.

    Assume now that $m \geq 3$ and that the statement holds for $m-1$.
    Let $N = \mathrm{R}(P_{n-2}, K_m) + \mathrm{R}(C_{n-1}, K_{m-1})$.
    Let $G$ be a complete graph of order $N$, and consider any red-blue coloring of the edges of $G$.
    We shall show that $G$ contains either a red $C_{n-1}$ or a blue $K_m$.
    Recall that $G_r$ and $G_b$ denote the graph induced by red edges and blue edges, respectively.

    We fix a vertex $v$ of $G$.
    Since $d_G(v)=\mathrm{R}(P_{n-2}, K_m) + \mathrm{R}(C_{n-1}, K_{m-1})-1$, either $|N_{G_r}(v)|\geq \mathrm{R}(P_{n-2}, K_m)$ or $|N_{G_b}(v)|\geq \mathrm{R}(C_{n-1}, K_{m-1})$.
    If $|N_{G_r}(v)|\geq \mathrm{R}(P_{n-2}, K_m)$, then either $G[N_{G_r}(v)]$ contains a blue $K_m$ or a red $P_{n-2}$, which together with $v$ forms a red $C_{n-1}$ in $G$.
    If $|N_{G_b}(v)|\geq \mathrm{R}(C_{n-1}, K_{m-1})$, then either $G[N_{G_b}(v)]$ contains a red $C_{n-1}$ or a blue $K_{m-1}$, which together with $v$ forms a blue $K_m$ in $G$.
    Thus, we have 
    \[
    \mathrm{R}(C_{n-1}, K_m) \leq \mathrm{R}(C_{n-1}, K_{m-1}) + \mathrm{R}(P_{n-2}, K_m).
    \]
    
    By the induction hypothesis and Lemma~\ref{lem:tree_complete_ramsey},
    \begin{align*}
        \mathrm{R}(C_{n-1}, K_m) &\leq \left( (m - 1) + \frac{n - 3}{2} (m - 1)(m - 2) \right) + ((n - 3)(m - 1) + 1) \\
        &= m + \frac{n - 3}{2} m(m - 1),
    \end{align*}
    which completes the proof.
\end{proof}

Now we prove Theorem~\ref{thm:k_ramsey_upper}.

\begin{proof}[Proof of Theorem~\ref{thm:k_ramsey_upper}]
    Let $s:=\mathrm{R}(C_{n-1},\overbrace{W_n, \ldots , W_n}^{k-1})-1$ and let $N_1=ks+2$.
    Consider any $k$-coloring of the edges of a complete graph $G$ with $N_1$ vertices.
    For a vertex $v$, since $d_G(v)=N_1-1=ks+1$, there exists a color $c\in \{1,\ldots ,k\}$ such that $|N_{G_c}(v)|\geq s+1=\mathrm{R}(C_{n-1},\overbrace{W_n, \ldots , W_n}^{k-1})$, where $G_c$ is a spanning subgraph of $G$ induced by all the edges of color $c$.
    Then, $G[N_{G_c}(v)]$ contains either a monochromatic $W_n$ or a monochromatic $C_{n-1}$ of color $c$, which together with $v$ forms a monochromatic $W_n$ in $G$.
    Thus, we have
    \[
    \mathrm{R}_k(W_n)\leq N_1=k\bigl(\mathrm{R}(C_{n-1},\overbrace{W_n, \ldots , W_n}^{k-1})-1\bigr)+2.
    \]

    Next, we show the second inequality in Theorem~\ref{thm:k_ramsey_upper}.
    Set $t:=\mathrm{R}_{k-1}(W_n)$ and $N_2:=\mathrm{R}(C_{n-1}, K_t)$.
    Consider any $k$-coloring of edges of a complete graph $H$ with $N_2$ vertices.
    Mapping the colors 1 to red and $\{2,\ldots ,k\}$ to blue, we obtain a red-blue coloring of the edges of $H$.
    By the definition of $N_2$, either a red $C_{n-1}$ or a blue $K_t$ exists.
    In the former case, $H$ contains a monochromatic $C_{n-1}$ of color 1.
    In the latter case, a set $S\subseteq V(H)$ induces a blue $K_t$, and it forces that there is a monochromatic $W_n$ of a color in $\{2,\ldots ,k\}$ in $H[S]\subseteq H$.
    Thus, we have
    \[
    \mathrm{R}(C_{n-1},\overbrace{W_n, \ldots , W_n}^{k-1})\leq N_2=\mathrm{R}(C_{n-1},K_{\mathrm{R}_{k-1}(W_n)}),
    \]
    and hence
    \[
    \mathrm{R}_k(W_n)\leq k\bigl(\mathrm{R}(C_{n-1},\overbrace{W_n, \ldots , W_n}^{k-1})-1\bigr)+2
    \leq k\bigl(\mathrm{R}(C_{n-1},K_{\mathrm{R}_{k-1}(W_n)})-1\bigr)+2.
    \]

    Finally, by applying Lemma~\ref{lem:cycle_complete_ramsey} to $\mathrm{R}(C_{n-1},K_{\mathrm{R}_{k-1}(W_n)})$, we infer that
    \begin{align*}
        \mathrm{R}_k(W_n) &\leq k\bigl(\mathrm{R}(C_{n-1},K_{\mathrm{R}_{k-1}(W_n)})-1\bigr)+2 \\
        &\leq k\left(\frac{n-3}{2}\mathrm{R}_{k-1}(W_n)\cdot (\mathrm{R}_{k-1}(W_n)-1)+\mathrm{R}_{k-1}(W_n)-1\right)+2.
    \end{align*}
    This completes the proof of Theorem~\ref{thm:k_ramsey_upper}.
\end{proof}

\section*{Acknowledgement}
Masaki Kashima is supported by JSPS KAKENHI grant number: JP25KJ2077. 
Yaping Mao is supported by the National Science Foundation of China grant number: 12061059.

\bibliographystyle{abbrv}
\bibliography{reference}

@book{bondy2008graph,
  title={Graph theory},
  author={Bondy, John Adrian and Murty, Uppaluri Siva Ramachandra},
  year={2008},
  publisher={Springer Publishing Company, Incorporated}
}

@article{brandt1998weakly,
  title={Weakly pancyclic graphs},
  author={Brandt, Stephan and Faudree, Ralph and Goddard, Wayne},
  journal={Journal of Graph Theory},
  volume={27},
  number={3},
  pages={141--176},
  year={1998},
  publisher={Wiley Online Library}
}

@article{burr1983generalizations,
  title={Generalizations of a {R}amsey-theoretic result of {C}hv{\'a}tal},
  author={Burr, Stefan A and Erd{\H{o}}s, Paul},
  journal={Journal of Graph Theory},
  volume={7},
  number={1},
  pages={39--51},
  year={1983},
  publisher={Wiley Online Library}
}

@article{chen2005ramsey,
  title={The {R}amsey numbers of paths versus wheels},
  author={Chen, Yaojun and Zhang, Yunqing and Zhang, Kemin},
  journal={Discrete Mathematics},
  volume={290},
  number={1},
  pages={85--87},
  year={2005},
  publisher={Elsevier}
}

@article{chvatal1977tree,
  title={Tree-complete graph {R}amsey numbers},
  author={Chv{\'a}tal, Vasek},
  journal={Journal of Graph Theory},
  volume={1},
  number={1},
  pages={93--93},
  year={1977},
  publisher={Wiley Online Library}
}

@article{dirac1952some,
  title={Some theorems on abstract graphs},
  author={Dirac, Gabriel Andrew},
  journal={Proceedings of the London Mathematical Society},
  volume={3},
  number={1},
  pages={69--81},
  year={1952},
  publisher={Oxford University Press}
}

@article{erdos1978cycle,
  title={On cycle-Complete graph ramsey numbers},
  author={Erd{\H{o}}s, Paul and Faudree, Ralph J and Rousseau, Cecil C and Schelp, Richard H},
  journal={Journal of Graph Theory},
  volume={2},
  number={1},
  pages={53--64},
  year={1978},
  publisher={Wiley Online Library}
}

@article{faudree1974all,
  title={All {R}amsey numbers for cycles in graphs},
  author={Faudree, Ralph J and Schelp, Richard H},
  journal={Discrete Mathematics},
  volume={8},
  number={4},
  pages={313--329},
  year={1974},
  publisher={Elsevier}
}

@article{faudree1993conjecture,
  title={A Conjecture of {E}rd{\H{o}}s and the {R}amsey Number ${R}({W}_6)$},
  author={Faudree, Ralph J and McKay, BD},
  journal={Journal of Combinatorial Mathematics and Combinatorial Computing},
  volume={13},
  pages={23--31},
  year={1993}
}

@book{graham1991ramsey,
  title={{R}amsey theory},
  author={Graham, Ronald L and Rothschild, Bruce L and Spencer, Joel H},
  year={1991},
  publisher={John Wiley \& Sons}
}

@article{greenwood1955combinatorial,
  title={Combinatorial relations and chromatic graphs},
  author={Greenwood, Robert E and Gleason, Andrew Mattei},
  journal={Canadian Journal of Mathematics},
  volume={7},
  pages={1--7},
  year={1955},
  publisher={Cambridge University Press}
}

@article{harborth1988all,
  title={All {R}amsey numbers for five vertices and seven or eight edges},
  author={Harborth, Heiko and Mengersen, Ingrid},
  journal={Discrete Mathematics},
  volume={73},
  number={1-2},
  pages={91--98},
  year={1988},
  publisher={Elsevier}
}

@article{hendry1989ramsey,
  title={{R}amsey numbers for graphs with five vertices},
  author={Hendry, George RT},
  journal={Journal of Graph Theory},
  volume={13},
  number={2},
  pages={245--248},
  year={1989},
  publisher={Wiley Online Library}
}

@article{karolyi2001generalized,
  title={Generalized and geometric {R}amsey numbers for cycles},
  author={K{\'a}rolyi, Gyula and Rosta, Vera},
  journal={Theoretical Computer Science},
  volume={263},
  number={1-2},
  pages={87--98},
  year={2001},
  publisher={Elsevier}
}

@article{mao2022ramsey,
  title={{R}amsey and {G}allai-{R}amsey number for wheels},
  author={Mao, Yaping and Wang, Zhao and Magnant, Colton and Schiermeyer, Ingo},
  journal={Graphs and Combinatorics},
  volume={38},
  number={2},
  pages={42},
  year={2022},
  publisher={Springer}
}

@article{radziszowski1994paths,
  title={Paths, cycles and wheels in graphs without antitriangles},
  author={Radziszowski, Stanislaw and Jin, Xia},
  journal={Centre for Discrete Mathematics and Computing: Australasian Journal of Combinatorics},
  volume={9},
  year={1994}
}

@article{rosta1973ramsey,
  title={On a {R}amsey-type problem of {J}. {A}. {B}ondy and {P}. {E}rd{\H{o}}s. II},
  author={Rosta, Vera},
  journal={Journal of Combinatorial Theory, Series B},
  volume={15},
  number={1},
  pages={105--120},
  year={1973},
  publisher={Elsevier}
}

@article{salman2007ramsey,
  title={On {R}amsey numbers for paths versus wheels},
  author={Salman, ANM and Broersma, Haitze J},
  journal={Discrete mathematics},
  volume={307},
  number={7-8},
  pages={975--982},
  year={2007},
  publisher={Elsevier}
}

@article{shi2010ramsey,
  title={{R}amsey numbers of long cycles versus books or wheels},
  author={Shi, Lingsheng},
  journal={European Journal of Combinatorics},
  volume={31},
  number={3},
  pages={828--838},
  year={2010},
  publisher={Elsevier}
}

@article{zhang2010ramsey,
  title={The {R}amsey numbers for cycles versus wheels of even order},
  author={Zhang, Lianmin and Chen, Yaojun and Cheng, TC Edwin},
  journal={European Journal of Combinatorics},
  volume={31},
  number={1},
  pages={254--259},
  year={2010},
  publisher={Elsevier}
}

\end{document}